\documentclass[11pt,oneside]{amsart}

  \usepackage{amssymb}
	\usepackage{epsfig}
\usepackage{fullpage}

\usepackage{epstopdf}
\usepackage{graphicx}
\usepackage{color}
\usepackage{mathptmx}   
\DeclareMathAlphabet{\mathcal}{OMS}{cmsy}{m}{n} 

\usepackage{hyperref}
\usepackage[alphabetic]{amsrefs}

     \theoremstyle{plain}
      \newtheorem{theorem}{Theorem}[section]
      \newtheorem*{theorem*}{Theorem}
       \newtheorem*{theoremA}{Theorem~\ref{thm:main}}
       \newtheorem*{theoremAdual}{Theorem~\ref{thm:dualnorm}}
       \newtheorem*{theoremB}{Theorem~\ref{thm:twistfamily}}

      \newtheorem{corollary}[theorem]{Corollary}
      
      \newtheorem{lemma}[theorem]{Lemma}

 \theoremstyle{definition}
      \newtheorem{remark}[theorem]{Remark}

    \newtheorem{question}[theorem]{Question}
    
      \newcommand{\R}{{\mathbb R}}
      
      \newcommand{\Z}{\mathbb Z}
	\newcommand{\N}{\mathbb N}
	\newcommand{\Q}{\mathbb Q}
	
	\newcommand{\calR}{\mathcal R}

      \newcommand{\nil}{\varnothing}
	
	\newcommand{\defn}[1]{\textbf{{#1}}}
	\newcommand{\bdry}{\partial}
	\newcommand{\boundary}{\partial}

	\newcommand{\st}[1]{\operatorname{st}}

	\newcommand{\nbhd}{\mathcal{N}}
	
	\newcommand{\wihat}[1]{\widehat{#1}}
	\renewcommand{\hat}[1]{\widehat{{#1}}}	
	
	\newcommand{\cut}{\setminus}
	
	\newcommand{\lk}{{{\ell k}}}

	\newcommand{\co}{\mskip0.5mu\colon\thinspace} 
	
	\newcommand{\wind}{{\rm{wind}}}
	\newcommand{\wrap}{{\rm{wrap}}}
	
	\newcommand{\im}{\operatorname{im}}

\newcommand{\spacing}{	  
\parskip 6.6pt
\parindent 0pt
}
   
   \usepackage{marvosym}
   \newcommand{\marka}{\ensuremath{*}}
   \newcommand{\markb}{\ensuremath{\dagger}}
   \newcommand{\markc}{\ensuremath{\ddagger}}
   \newcommand{\markd}{\ensuremath{\circledast}}

\spacing

\begin{document}


   \title[]{Dehn filling and the Thurston norm}

\author{Kenneth L.\ Baker and Scott A.\ Taylor}

 \begin{abstract}
For a  compact, orientable, irreducible $3$--manifold with toroidal boundary that is not the product of a torus and an interval or a cable space, each boundary torus has a finite set of slopes such that, if avoided, the Thurston norm of a Dehn filling behaves predictably. More precisely, for all but finitely many slopes,  the Thurston norm of a class in the second homology of the filled manifold  plus the so-called winding norm of the class will be equal to the Thurston norm of the corresponding class in the second homology of the unfilled manifold. This generalizes a result of Sela and is used to answer a question of Baker-Motegi concerning the Seifert genus of knots obtained by twisting a given initial knot along an unknot which links it.
     \end{abstract}

   \date{\today}
\thanks{}

   \maketitle


\section{Introduction}
How does the Thurston norm behave under Dehn filling?

Let $N$ be a compact, orientable $3$--manifold with toroidal boundary and let $T \subset \boundary N$ be a particular component. Consider the Dehn fillings $N_T(b)$ along slopes $b$ in $T$.  For each slope $b$ in $T$, the Dehn filling induces a natural inclusion of $N$ into $N_T(b)$ that induces the monomorphism
	\[ \iota_b \colon H_2(N, \bdry N - T) \to H_2(N_T(b),\bdry N_T(b))\]
defined as follows.   If $z \in H_2(N, \bdry N - T)$ is represented by a properly embedded surface $S$ in $N$ with $\bdry S \cap T = \nil$, then $\iota_b(z)=\hat{z}$ is also represented by $S$ under the inclusion. Consequently,  
\[x(z) \geq x(\hat{z})  \tag{\marka}\] 
on the Thurston norms of classes $z \in H_2(N, \bdry N - T)$ and $\iota_b(z)=\hat{z} \in H_2(N_T(b), \bdry N_T(b))$.  

Gabai and Sela both address when Inequality (\marka) is an equality. Gabai shows that for a fixed class $z \in H_2(N, \bdry N - T)$, $x(z) = x(\hat{z})$ for all except at most one slope $b$ in $T$ \cite[Corollary 2.4]{G2}. Sela extends this result showing that the equality $x(z) = x(\hat{z})$ holds for every class $z \in H_2(N, \bdry N-T)$ and induced class $\hat{z} \in H_2(N_T(b), \bdry N_T(b))$ for all Dehn fillings except along a finite number of slopes $b$ in $T$ \cite[Theorem 3]{Sela}\footnote{Sela uses \cite[Theorem 1.8]{G2} which required an atoroidality hypothesis.  However \cite[Corollary 2.4]{G2} can be used instead to avoid such an additional hypothesis.  Lackenby discusses such atoroidality hypotheses in the Appendix to \cite{LackenbyDSOK}.}.

In this article we extend consideration to all classes in $H_2(N, \bdry N)$.  To do so, for each slope $b$ in $T$  we consider the restriction of the Dehn filling $N_T(b)$  to $N$ rather than the inclusion of $N$ into $N_T(b)$.    Restriction gives a monomorphism 
	\[\rho_b \colon H_2(N_T(b), \bdry N_T(b)) \to  H_2(N, \bdry N)\]
defined as follows.    If $\hat{z} \in H_2(N_T(b), \bdry N_T(b))$ is represented by a properly embedded surface $\hat{S}$ that is transverse to $K_b$, then  $\rho_b(\hat{z}) = z$ is represented by $S = \hat{S} \cap N$.  Here, and throughout, we take $K_b \subset N_T(b)$ to be the core of the filling with tubular neighborhood $\nbhd(K_b)$ so that $N = N_T(b) - \nbhd(K_b)$, and we orient $K_b$ and its meridian $b$ so that $b$ links $K_b$ positively.  The algebraic intersection number with the core $K_b$ is a linear form on homology, so its absolute value is a pseudo-norm.  That is,  the pseudo-norm \defn{winding number} of $K_b$ about a homology class $\hat{z} \in H_2(N_T(b), \bdry N_T(b))$ is defined to be
 \[\wind_{K_b} ( \hat{z}) = |[K_b] \cdot \hat{z}|.\]

The winding number enables the following extension of Inequality (\marka), whose proof is given in Section \ref{sec:Thurstonnormintro}.  
\begin{lemma}\label{lem:windineq}
	Let $N$ be a compact, orientable, irreducible $3$--manifold whose boundary is a union of tori.  Let $T$ be a component of $\bdry N$ and let $b$ be a slope in $T$.
If $N_T(b)$ has no $S^1 \times D^2$ or $S^1 \times S^2$ summands, then for all classes  $\hat{z} \in H_2(N_T(b), \bdry N_T(b))$,   
\[ x(z) \geq  x(\hat{z}) + \wind_{K_b}(\hat{z}) \tag{\markb}\]
where $\rho_b(\hat{z}) = z$.
\end{lemma}

Our main goal in this paper is to address when Inequality (\markb) is an equality, i.e. when
\[x(z) = x(\hat{z}) + \wind_{K_b}(\hat{z}) \tag{\markc}.\]
 
 For convenience, if there exists a class $\hat{z} \in H_2(N_T(b), \bdry N_T(b))$ for which Equality (\markc) fails,  then we say the slope $b$ is a \defn{norm-reducing} slope, the class $z = \rho_b(\hat{z}) \in H_2(N, \bdry N)$ is a \defn{norm-reducing} class with respect to the norm-reducing slope $b$, and  the class $\hat{z} \in H_2(N_T(b), \bdry N_T(b))$ is a \defn{norm-reducing} class with  respect to the knot $K_b$.  

\begin{theoremA}
Let $N$ be a compact, connected, orientable, irreducible $3$--manifold whose boundary is a union of tori. Then either 
\begin{enumerate}
\item $N$ is a product of a torus and an interval,
\item $N$ is a cable space, or 
\item for each torus component $T \subset \bdry N$ there is a finite set of slopes $\calR=\calR(N,T)$ in $T$ such that if $b \not \in \calR$ then $b$ is not norm-reducing.
 \end{enumerate}
\end{theoremA} 

In Corollary~\ref{cor:boundsondegenerateslopes} we obtain a bound on the size of $\calR(N,T)$ in terms of the Thurston norms of two integral classes of two different fillings and the distance between the two filling slopes. Since $\wind_{K_b}(\hat{z}) = 0$ when $\rho_b(\hat{z}) \in H_2(N, \bdry N-T)$, Theorem \ref{thm:main} generalizes Sela's result (with the additional assumption that $N$ is irreducible). Sela also explicitly bounds, by the number of faces of the Thurston norm ball of $H_2(N, \bdry N-T)$, the number of slopes $b$ for which Equation (\markc) may fail for classes $z=\rho_b(\hat{z}) \in H_2(N, \bdry N-T)$  when $\wind_{K_b}(\hat{z}) = 0$. We appeal to his result to handle the classes in $H_2(N, \boundary N - T)$.

In the same vein as Gabai's and Sela's results, Lackenby \cite[Theorem 1.4b]{Lackenby} (under additional hypotheses and a change of notation\footnote{In Lackenby's paper, see Assumptions 1.1 and Remark 1.3. To convert the notation from ours to Lackenby's make the following changes: $\gamma =\nil$, $M' \to M$, $K_a \to L$, $N \to M - \operatorname{int}(N(L))$, $\hat{Q} \to F$. The class whose norm is reduced is called $z_1$ by Lackenby.
})  showed that if $\hat{Q}$ is a compact connected surface in $M' = N_T(a)$ which cannot be isotoped to be disjoint from $K_a$ and if there is a norm-reducing class under a filling of slope $b$ with $\Delta=\Delta(a,b) \geq 2$, then $\hat{Q}$ can be isotoped so that
\[
|K_a \cap \hat{Q}|(\Delta - 1) \leq -\chi(\hat{Q}).
\]
If, in Lackenby's setup, $\hat{Q}$ is taken to be a taut representative of a non-zero class $\hat{y} \in H_2(M', \boundary M')$, then we have (after rearranging the inequality):
\[
\Delta \leq 1 + \frac{x(\hat{y})}{|K_a \cap \hat{Q}|}.
\]
Our Corollary \ref{cor:twofillings}, gives a version of this result for the situation when $H_2(N,\boundary N)$, and not just $H_2(N, \boundary N  - T)$, has a norm-degenerating class with respect to the slope $b$.

In addition to considering a fixed component $T$ of $\boundary N$ and studying the dependency of the Thurston norm on the filling slope, we can also consider a 3-manifold $M$ and consider how the Thurston norm of manifolds $M'$ obtained by surgery on an oriented knot $K$ in $M$ depends on the dual Thurston norm $x^*([K])$ of the class $\alpha = [K] \in H_1(M;\Z)$.

\begin{theoremAdual}
Let $M$ be a compact, orientable $3$--manifold whose boundary is a union of tori, $\Delta \in \N$, and $\alpha \in H_1(M;\Z)$. Assume that every sphere, disc, annulus, and torus in $M$ separates. If 
\[ (\Delta -1) x^*(\alpha) >1 \]
then every irreducible, $\bdry$--irreducible $3$--manifold obtained by a Dehn surgery of distance $\Delta$ on a knot $K$ representing $\alpha$ has no norm-reducing classes with respect to the knot which is surgery dual to $K$.
\end{theoremAdual}

The contrapositive is also a useful formulation, as it shows that knots resulting from non-longitudinal surgery on a knot with a norm-reducing class have bounded dual norm.

Finally, we give an application to the genus of knots in twist families. A {\em twist family} of knots $\{K_n\}$ is obtained by performing $-1/n$--Dehn surgery on an unknot $c$ that links a given knot $K=K_0$.   When $\lk(K,c)=0$, it is a fundamental consequence of \cite[Corollary 2.4]{G2} that $g(K_n)$ is constant for all integers $n$ except at most one where the genus decreases.
Using the multivariable Alexander polynomial, the first author and Motegi showed that if $|\lk(K,c)| \geq 2$, then $g(K_n) \to \infty$ as ${n\to \infty}$ \cite{bakermotegi}.  When $|\lk(K,c)|=1$, this fails if $c$ is a meridian of $K$ since $K_n = K$ for all $K$.  Here we answer \cite[Question 2.2]{bakermotegi} by showing this is the only exception.

\begin{theoremB}
If $\omega=|\lk(K,c)|>0$, then $\lim_{n\to \infty} g(K_n) =\infty$ unless $c$ is a meridian of $K$.
\end{theoremB}

\section{Preliminaries}

\subsection{Notation and conventions}  The following notation is used throughout the article. 
We take $N$ to be a compact, connected, irreducible oriented $3$--manifold where $\boundary N$ is a non-empty union of tori and focus upon a particular component $T \subset \boundary N$.  Given two slopes $a,b \subset T$, we set the results of Dehn filling $N$ along these slopes to be the two $3$--manifolds $M=N_T(b)$ and $M'=N_T(a)$.  Furthermore we let $K=K_b \subset M$ and $K'=K_a \subset M'$ denote the core knots of the two filling solid tori.   

The \defn{distance} $\Delta = \Delta(a,b)$ between  two slopes $a,b \subset T$ is the minimal number of points of intersection between simple closed curves in $T$ representing $a$ and $b$.

Given a surface $S$ properly embedded in $N$, the union of the boundary components of $S$ in $T$ is $\bdry_T S = \bdry S \cap T$.  If the slope of each component of $\bdry_T S$ in $T$ is $b$ (as an unoriented curve), then we set $\hat{S} \subset M$ to be the surface obtained by capping off the components of $\bdry_T S$ with meridian discs of the filling solid torus. Observe that by construction, 
$|K \cap \hat{S}|=|\bdry_T S|$. 

In this article, a \defn{lens space} is a closed $3$--manifold with a genus $1$ Heegaard splitting other than $S^3$ and $S^1 \times S^2$.  In particular, the fundamental group of a lens space is a non-trivial, finite, cyclic group.

\subsection{Thurston norm}\label{sec:Thurstonnormintro}
Thurston introduced two norms on the homology groups of a compact, orientable $3$--manifold $W$ \cite{Thurston},  now commonly known as the \defn{Thurston norm} and the \defn{dual Thurston norm}:
\[x\colon H_2(W, \boundary W; \R) \to [0, \infty) \quad \mbox{ and }\quad x^* \colon H_1(W; \R) \to [0, \infty)\]
which we may write as $x_W$ and $x^*_W$ to emphasize the $3$--manifold $W$.

On an integral class $\sigma \in H_2(W, \bdry W; \Z)$, 
the Thurston norm is defined  by
\[
x(\sigma) = \min\limits_S \sum_{i = 1}^n \max\{0, -\chi(S_i)\}
\]
where the minimum is taken over all embedded surfaces $S$ representing $\sigma$ with connected components $S_1, \hdots, S_n$. The function $x$ is linear on rays and convex. These properties enable it to be extended first to rational homology classes and then to real homology classes.   

In general, the function $x$ is only a pseudo-norm; $x$ is a norm if $W$ contains no non-separating sphere, disc, torus, or annulus.  Nevertheless, $x$ is reasonably well behaved even in the presence of non-separating tori and annuli, it is non-separating spheres and discs that complicate the norm:
\begin{quotation}
{\em 
If an integral class $\sigma \in H_2(W, \bdry W;\Z)$ cannot be represented by a surface with a non-separating sphere or disc component, then $x(\sigma)$ is just the minimum of $-\chi(S)$ among surfaces representing $\sigma$.
}
\end{quotation}
It is for such integral classes that Inequality (\markb) holds. Assuming $W$ has no $S^1 \times S^2$ or $S^1 \times D^2$ summand ensures this is the case for all classes, as does the more heavy-handed assumption that $W$ is irreducible and $\bdry$--irreducible. In particular, we can now prove Lemma \ref{lem:windineq}.

\begin{proof}[Proof of Lemma \ref{lem:windineq}]
Recall that $N$ is a compact, orientable, irreducible 3--manifold with $\boundary N$ the union of tori and $T \subset \boundary N$ a component. Let $b$ be a slope in $T$ and assume that $N_T(b)$ has no $S^1 \times D^2$ or $S^1 \times S^2$ summands. Let $\boundary_T \co H_2(N, \boundary N) \to H_1(T)$ be the boundary map restricted to $T$. We will show that for all classes $\hat{z} \in H_2(N_T(b), \boundary N_T(b))$,
\[ x(z) \geq  x(\hat{z}) + \wind_{K_b}(\hat{z}) \tag{\markb}.\]

As usual, it suffices to prove the inequality for integral classes. In which case, there exists a properly embedded oriented surface $S \subset N$ such that $S$ has no separating component, $[S] = z$, and all components of $\boundary_T S$ are coherently oriented curves, each of slope $b$, and $x(S) = x(z)$. If some component of $S$ is a sphere or disc, then it would persist into $N_T(b)$ as a non-separating sphere or disc, contrary to our hypotheses. Hence $S$ has no sphere or disc component and $x(S) = -\chi(S)$.

Cap off the components of $\boundary_T(S)$ in $N_T(b)$ with discs to obtain the surface $\hat{S}$. Observe that 
\[ |\boundary_T S| = |\hat{S} \cap K_b| = \wind_{K_b}(\hat{z})\]
 since the components of $\boundary_T S$ are coherently oriented. Since $M$ contains no non-separating sphere or disc, $-\chi(\hat{S}) \geq x(\hat{z})$. Consequently,
\[
x(z) = -\chi(S) = -\chi(\hat{S}) + \wind_{K_b}(\hat{z}) \geq x(\hat{z}) + \wind_{K_b}(\hat{z}).
\]
\end{proof}

Finally, on a class $\alpha \in H_1(W;\R)$, the dual Thurston norm is defined by 
\[
x^*(\alpha) = \sup \limits_{x(\sigma) \leq 1} |\alpha \cdot \sigma|
\]
where $\cdot$ denotes the intersection product. The function $x^*\co H_1(W; \R) \to [0, \infty)$ is continuous.

\subsection{Wrapping numbers}

Having defined the winding number, we now turn to wrapping number. A  compact, oriented, properly embedded surface $S$ in a $3$--manifold $W$ is \defn{taut} (or $\nil$--taut) if it is incompressible (i.e.\ does not admit a compressing disc), and minimizes the Thurston norm among embedded surfaces representing the class $[S, \boundary S] \in H_2(W, \boundary S)$ \cite[Def. 1.2]{S1}.  Observe that if a surface $S \subset N$ is taut and has the property that $x(S) = x([S])$, then the surface $S'$ obtained by discarding all separating components of $S$ (which are necessarily spheres, discs, annuli, and tori) is also taut and has the properties that $[S] = [S'] \in H_2(N,\boundary N)$ and $x(S') = x([S]) = x([S'])$.

We define the \defn{wrapping number} of $K$ about an integral homology class $\hat{z} \in H_2(M, \boundary M;\Z)$ to be 
\[ \wrap_{K}(\hat{z}) = \min_{\hat{S}} |K \cap \hat{S}| \]
where the minimum is taken over all {\em taut} representatives $\hat{S}$ of $\hat{z}$. 

	\begin{quotation}
		{\em 
	Since discarding separating components of $\hat{S}$ will not increase $|K \cap \hat{S}|$, we will henceforth assume that whenever we discuss a taut surface realizing the Thurston norm of a homology class in the second homology group of a 3-manifold relative to the boundary of that 3-manifold, we have discarded all separating components.
		}
	\end{quotation}

We may extend the wrapping number to  $H_2(M, \bdry M;\Q)$.   Assume $\hat{S}$ is a taut surface realizing $\wrap_{K}(\hat{z})$ for an integral class $\hat{z} \in H_2(M, \bdry M;\Z)$.  Then, following \cite[Lemma 1]{Thurston}, $n$ parallel copies of $\hat{S}$ is a taut surface realizing $\wrap_{K}(n\, \hat{z})= n\, \wrap_{K}(\hat{z})$ for positive integers $n$. Thus for a rational class $\hat{q}$ we define $\wrap_{K}(\hat{q}) = \tfrac{1}{n} \wrap_{K}(n\,\hat{q})$ where $n$ is a positive integer such that $n\hat{q}$ is an integral class. Since algebraic intersection numbers give lower bounds for geometric intersection numbers, $\wrap_K(\hat{q}) \geq \wind_K(\hat{q})$ for all $\hat{q} \in H_2(M,\boundary M; \Q)$. Observe that if $M$ has no norm-reducing classes with respect to $K$, then $\wrap_{K} = \wind_{K}$ is a pseudo-norm. However, we believe that, in general, the triangle inequality will not hold for $\wrap_{K}$.
	\begin{question}
		Must the wrapping number satisfy the triangle inequality?
	\end{question}

A class $\hat{z} \in H_2(M, \boundary M)$ is \defn{exceptional} with respect to a knot $K$ \cite{T} if the winding number and wrapping number are not equal; that is $\hat{z}$ is exceptional with respect to $K$ if
\[
\wind_{K}(\hat{z}) < \wrap_{K}(\hat{z}).
\]
This definition takes root in the practical difference between the Thurston norm and Scharlemann's $\beta$--norm.  As discussed in \cite{T}, a class $\hat{z}$ is {\em exceptional} with respect to $K$ if and only if no representative of $\hat{z}$ is both $\nil$--taut and $K$--taut.   (Here, $K$ is playing the role of $\beta$.  See \cite{S1} for the definitions of the $\beta$--norm and $\beta$--taut surfaces.)

For our present purposes, we observe that {\em norm-reducing} classes and {\em exceptional} classes are equivalent in the absence of non-separating spheres and discs.  This allows us to parlay technical results about exceptional classes into results about norm-reduction.
\begin{lemma}\label{lem:exceptionalisnormdegen}
Suppose that $M$ contains no non-separating sphere or disc. Then, with respect to a knot $K$ in $M$, a class $\hat{z} \in H_2(M, \boundary M)$ is exceptional if and only if it is norm-reducing.
\end{lemma}
\begin{proof}
Assume $M=N_T(b)$ where $K=K_b$.  For a class $\hat{z} \in H_2(M, \bdry M)$, let $z = \rho_b(\hat{z}) \in H_2(N, \bdry N)$.

First, we claim that if $S$ is a taut representative of a class $[S] \in \im \rho_b$, then 
\[
x([S]) = x(S) = -\chi(S).
\]
To see this, let $S \subset N$ be taut and have each component of $\boundary_T S$ of slope $b$. By definition, $x([S]) = x(S)$. Suppose that $x(S) \neq -\chi(S)$. Then $S$ contains a component $P$ which is a sphere or disc. Since $S$ is taut, $P$ is non-separating. Capping off $\boundary_T P$ in $M$, if necessary, creates a non-separating sphere or disc in $M$, contrary to hypothesis.

We now embark on the proof. The claim is trivially satisfied for the 0 class, so assume that $0 \neq \hat{z} \in H_2(M, \bdry M;\Z)$ is not an exceptional class for $K$.  Then there is a taut representative $\hat{S} \subset M$ of $\hat{z}$ for which $\wrap_K(\hat{S}) = \wind_K(\hat{S})$. Thus
\[
x_N(z) \leq x_N(S) = -\chi(S) = -\chi(\hat{S}) + \wind_K(\hat{S}) = x_M(\hat{z}) + \wind_K(\hat{z}) \leq x_N(z)
\]
where the last inequality is due to Inequality (\markb).  Consequently $x_M(\hat{z}) + \wind_K(\hat{z}) = x_N(z)$, and thus $\hat{z}$ is not norm-reducing with respect to $K$.

Conversely, assume that $\hat{z} \in H_2(M, \boundary M)$ is exceptional with respect to $K$ so that $\wrap_K(\hat{z}) > \wind_K(\hat{z})$.
Let $S$ be a taut surface in $N$ representing $z$, and let $\hat{S} \subset M$ be the result of capping off $\boundary_T S$ with discs so that $[\hat{S}] = \hat{z}$. Then
\[
x_N(z) = -\chi(S) = -\chi(\hat{S}) + |\hat{S} \cap K| > x_M(\hat{z}) + \wind_K(\hat{z})
\]
because $|\hat{S} \cap K| \geq |\hat{S} \cdot K| = \wind_K(\hat{z})$ and $-\chi(\hat{S}) \geq x_M(\hat{z})$.
Thus, $\hat{z}$ is norm-reducing with respect to $K$.
\end{proof}

\subsection{Multi-$\bdry$-compressing discs}
As is often the case in studies of Dehn filling, we will want use a surface $\hat{Q}$ in one filling $M' = N_T(a)$ of $N$ to say something useful about a different filling $M = N_T(b)$. For us, the surface $\hat{Q}$ will be most useful if it has no ``multi-$\boundary$-compressing disc.''

Suppose that $\wihat{S} \subset M'=N_T(a)$ is a surface transversally intersecting $K' \subset M'$ non-trivially. A \defn{multi-$\boundary$-compressing disc} for $\wihat{S}$ (with respect to $K'$) is a disc $D \subset N$ such that there is a component $A \subset T  - S$ such that:
\begin{itemize}
\item The interior of $D$ is disjoint from $\boundary N \cup S$
\item The boundary of $D$ is a simple closed curve lying in $S \cup A$
\item After orienting $\boundary D$, $\boundary D \cap A$ is a non-empty, coherently oriented collection of spanning arcs of $A$. 
\end{itemize}

Given a multi-$\boundary$-compressing disc $D$ for $\wihat{S}$, then we may create a new surface $\wihat{S}'$  that is homologous to $\hat{S}$ but intersects $K'$ in two fewer points: that is, $[\wihat{S}] = [\wihat{S}'] \in H_2(M', \boundary M')$ and $|\wihat{S}' \cap K'| = |\wihat{S} \cap K'| - 2$. We create $\wihat{S}'$ by removing the open regular neighborhood of two points of $K' \cap \wihat{S}$, attaching the annulus $A$ (from the definition of ``multi-$\boundary$-compressing disc'') and then compressing using $D$. 

The next lemma allows us to know when we have a surface without a multi-$\boundary$-compressing disc.

\begin{lemma}\label{no gabai disc}\
\begin{itemize}
\item Suppose that $\wihat{S} \subset M'$ is a sphere transverse to $K'$ such that $S = \wihat{S} \cap N$ is incompressible and not $\boundary$--parallel.  
 Then either $M'$ has a lens space summand or $\wihat{S}$ does not have a multi-$\boundary$-compressing disc with respect to $K'$.   
 
\item Suppose that $\wihat{S} \subset M'$ is a disc transverse to $K'$ such that $S = \wihat{S} \cap N$ is incompressible.  
 Then either $M'$ has a lens space summand or $\wihat{S}$ does not have a multi-$\boundary$-compressing disc with respect to $K'$.  
 
\item Suppose that $\wihat{S} \subset M'$ is a taut representative of some non-zero class in $H_2(M', \boundary M';\Z)$ and that, out of all such taut surfaces representing that class, $\wihat{S}$ minimizes $|\wihat{S} \cap K'|$. 
Then either  $M'$ contains a non-separating sphere or disc or $\wihat{S}$ does not have a multi-$\boundary$-compressing disc with resepect to $K'$.
\end{itemize}
\end{lemma}

\begin{proof}
Suppose that $\wihat{S} \subset M'$ is a surface transverse to $K'$, such that $S$ is incompressible and not $\boundary$-parallel.  If $K'$ is disjoint from $\hat{S}$, then trivially there is no multi-$\bdry$-compressing disc.  Hence we further assume $K'$ transversally intersects $\hat{S}$ non-trivially.

Suppose that $D$ is an oriented multi-$\bdry$-compressing disc for $\hat{S}$. Then there is an annulus component $A \subset T\cut S$ such $\bdry D \cap A$ is a non-empty collection of coherently oriented spanning arcs of $A$. Let $\hat{R}$ be the surface in $M'$ obtained from isotoping $S \cup A \subset N$ with support in a neighborhood of $A$ to be properly embedded in $N$ and then capping off the boundary components in $T$ with meridional discs of the filling solid torus; i.e. $\hat{R}$ is the result of tubing $\hat{S}$ along a particular arc of $K' \cut \hat{S}$. A further slight isotopy makes $\hat{R}$ disjoint from $\hat{S}$.  

Now let $\wihat{S}'$ be the result of compressing $\wihat{R}$ using $D$, and slightly isotoping to be disjoint from $\hat{R}$. Observe that $-\chi(\wihat{S}') = -\chi(\wihat{S})$ and that there is a natural bijection between the components of $\wihat{S}$ and $\wihat{S}'$. 

First assume $\hat{S}$ is a sphere.  Then $\hat{S}'$ must also be a sphere.
If $\boundary D$ runs just a single time across $A$,  then $D$ provides a $\bdry$--compression for $S$ in $N$.  Since $N$ is irreducible, either $S$ is compressible or $S$ is a $\boundary$-parallel annulus contrary to hypothesis. If $\boundary D$ runs multiple times across $A$, then $\wihat{S}$ and $\wihat{S}'$ cobound a $3$--manifold $W$ in which $\hat{R}$ is a genus $1$ Heegaard surface.  Because  $\hat{S}$ and $\hat{S}'$ are both spheres, $W$ is a twice-punctured lens space of finite order $|\bdry D \cap A|>1$.  The complement of a neighborhood of an embedded arc in $W$ that connects both components of $\bdry W$ is therefore a non-trivial lens space summand of $M'$.  

When $\hat{S}$ is a disc, we similarly obtain that $\hat{S}'$ is also a disc.  Along with an annulus in $\bdry M'$, the discs $\hat{S}$ and $\hat{S}'$ bound a punctured lens space $W$ in which $\hat{R}$ is a punctured Heegaard torus.  Again, this lens space has finite order $|\bdry D \cap A|$ which is non-trivial since $\hat{S}$ is incompressible.  Hence $W$ is a lens space summand of $M'$.

Now assume that $\wihat{S}$ is a taut representative of a class in $H_2(M', \boundary M'; \Z)$.  If $\hat{S}$ has a sphere, then the component  must be non-separating since $\hat{S}$ is taut.  So we may further assume $\hat{S}$ is not a sphere.  By construction, the surface $\wihat{S}'$ represents the same class, has the same euler characteristic, and intersects $K'$ two fewer times than does $\wihat{S}$. Furthermore, since every component of $\hat{S}$ is non-separating, every component of $\hat{S}'$ is also non separating.   If $\wihat{S}'$ is not taut, then since it is homologous to the taut surface $\hat{S}$ and is also Thurston norm minimizing for this homology class, it must have a compressible component that is a non-separating torus or annulus.  Compressing this torus or annulus creates a non-separating sphere or disc in $M'$.
\end{proof}

\section{A key theorem of Taylor}\label{sec:taylor}
In \cite{T}, the second author develops some classical results (\cite[Application III]{S1} and \cite{S2}) from Scharlemann's combinatorial version \cite{S1} of  Gabai's sutured manifold theory  \cite{G1, G2, G3}  in terms of surgeries on knots with exceptional classes.  Here we adapt a key technical theorem for our purposes.
 
\begin{theorem}[{Cf.\ \cite[Theorem 3.14]{T}}]\label{Previous Sutured Thm}
Assume that $N$ is irreducible and $\boundary$--irreducible. Let $a,b$ be two distinct slopes in $T \subset \bdry N$. 
Suppose that $M=N_T(b)$ is not a solid torus, has no proper summand which is a rational homology sphere, and $H_2(M, \boundary M) \neq 0$.
Suppose that $M'=N_T(a)$ contains a properly embedded, compact, orientable surface $\wihat{Q} \subset M'$  that transversally intersects $K'$ non-trivially, does not have a multi-$\bdry$-compressing disc for $K'$, and restricts to an incompressible surface\footnote{We use the convention that any sphere component of an incompressible surface does not bound a ball, and any disc component is not $\bdry$--parallel.}
 $Q = \hat{Q} \cap N$ in $N$.
 
If
\[
-\chi(\hat{Q}) < |\hat{Q} \cap K'|(\Delta(a,b) - 1),
\]
then $M$ is irreducible and $H_2(M,\boundary M)$ has no exceptional classes with respect to $K$.
\end{theorem}

For the proof, we content ourselves with explaining how the statement follows from \cite[Theorem 3.14]{T}. We assume familiarity with the basic definitions regarding $\beta$--taut sutured manifold technology from \cite{S1} (see also \cite{T}).

\begin{proof}
Our notation is very similar to that of \cite{T}, except that we are using $K$  as the core knot of the filling $M=N(b)$ instead of $\beta$   and we consider classes $\hat{y} \in H_2(M,\boundary M)$ rather than classes $y$.

Our hypotheses immediately imply Conditions (1) and (3) of \cite[Theorem 3.14]{T}.
Since $N$ is irreducible and $\bdry$--irreducible, we may consider it as a taut sutured manifold $(N,\nil, \nil)$, considering $\bdry N$ as toroidal sutures. The filling $M=N_T(b)$ induces a sutured manifold $(M, \nil, K)$ that is then a $K$--taut sutured manifold, providing Condition (2). 

Since $\hat{Q} \cap K' \neq \nil$ and the curves of $\bdry_T Q$ have slope $a$, the boundary of $Q$ is not disjoint from the slope $b$ in $T$. Sphere components of $\hat{Q}$ that are disjoint from $K'$ are the sphere components of $Q$; however, since the irreducibility of $N$ implies that any sphere component of $Q$ must bound a ball in $N$, the incompressibility of $Q$ prohibits the existence of such sphere components.
Furthermore, no component of $Q$ is a disc with essential boundary since $N$ is $\bdry$--irreducible and no component of $Q$ is a disc with inessential boundary due to the incompressibility of $Q$ and irreducibility of $N$.  Thus Condition (4) is satisfied.

We may now apply \cite[Theorem 3.14]{T}.  Our hypothesis that $M$ has no proper summand that is a rational homology sphere immediately rules out  Conclusion (4) of \cite[Theorem 3.14]{T}.  We proceed to show that Conclusions (3) and (2) also fail and that Conclusion (1) implies our stated result.

In the terminology of \cite[Section 7]{S1} and \cite[Section 2.2]{T}, the surface $Q$ is a {\em parameterizing surface} for the sutured manifold $(M, \nil, K)$. By definition (again, see \cite[Definition 7.4]{S1} and \cite[Section 2.2]{T}), its {\em index} $I(Q)$ is given by
\[
I(Q) = -2\chi(Q)
\]
since (i) there are no annular sutures on $\bdry M$ and (ii) $K$ is a knot (rather than a collection of properly embedded arcs). Without loss of generality, we may assume that the slope $b$ has been isotoped in $T$ to intersect $\bdry Q$ minimally. Thus, $|\boundary Q \cap b|$ is equal to $\Delta(a,b)|\wihat{Q} \cap K'|$. Our assumed inequality on the Euler characteristic of $\wihat{Q}$ can then be rearranged to yield
\[
I(Q) < 2|\boundary Q \cap b|. 
\]
Hence, Conclusion (3) of \cite[Theorem 3.14]{T} does not hold.

A {\em Gabai disc} for $Q$ is a disc $D$ embedded in $M$ that $K$ non-trivially and coherently intersects, such that its restriction to $N$ is transverse to $Q$ and $|Q \cap \bdry D| < \Delta(a,b) |\bdry_T Q|$. 
It is shown in \cite{CGLS} (though without the language of Gabai discs), and further explained in \cite{S2} and \cite{T}, that a Gabai disc will contain a Scharlemann cycle. As $Q$ is incompressible and $N$ is irreducible, the interior of the Scharlemann cycle can be isotoped to be a multi-$\bdry$-compressing disc for $\wihat{Q}$. See \cite[Section 4]{T} for more details. (Although observe that \cite[Lemma 4.3]{T} neglected to consider possible circles of intersection between the interior of the Scharlemann cycle and $Q$. We have added the incompressibility hypotheses to $Q$ to deal with this.) Since we are assuming that $\wihat{Q}$ has no multi-$\bdry$-compressing disc, Conclusion (2) of \cite[Theorem 3.14]{T} does not hold.

Consequently, the Conclusion (1) of \cite[Theorem 3.14]{T} holds. Hence, given any non-zero class $\hat{y} \in H_2(M, \bdry M; \Z)$, there is a $K$--taut hierarchy of $(M, \nil, K)$ which is also $\nil$--taut such that the first decomposing surface $\hat{S} \subset M$ represents $\hat{y}$.  In particular, since sutured manifold decompositions yields a taut sutured manifold only if the decomposing surface is taut, the $K$--tautness and $\nil$--tautness of the hierarchy implies the surface $\hat{S}$ must be both $K$--taut and $\nil$--taut (see e.g.\  \cite[Definition 4.18]{S1},  \cite[Section 2]{S2}, \cite[Lemma 3.5 and Section 4]{G1}).  
Since $(M, \nil, \nil)$ is $\nil$--taut, $M$ is irreducible. By the definition of $K$--taut, the knot $K$ always intersects $\hat{S}$ with the same sign. That is, $\wind_K(\hat{S}) = \wrap_K(\hat{S})$. Since $\hat{S}$ is $\nil$--taut, this implies that $\hat{y}$ is not an exceptional class. Since this holds true for all non-zero classes in $H_2(M, \boundary M; \Z)$, so there are no exceptional classes in $H_2(M, \boundary M;\Z)$ with respect to $K$.
\end{proof}

\section{The Thurston norm and dual norm under Dehn filling} 

\subsection{The Thurston norm}

\begin{theorem}\label{thm:BoundingThurstonNorm}
Suppose that $N$ is irreducible and $\boundary$--irreducible. Also assume that $M=N_T(b)$ is not a solid torus and has no proper rational homology sphere summand and that either $M$ is reducible or that  $H_2(M, \bdry M)$ has an exceptional class with respect to $K$. Then all of the following hold for $M'=N_T(a)$:
\begin{itemize}
\item Either $M'$ has a lens space summand or 
	\begin{itemize}
		\item $M'$ is irreducible and $\bdry$--irreducible, and 
		\item  $K' \subset M'$ is mp-small; that is, there is no essential, connected, properly embedded planar surface $Q \subset N$ such that $\bdry Q = \bdry_T Q \neq \nil$ and each component of $\bdry Q$ has slope $b$ in $T$.
	\end{itemize}

\item For every  $\hat{y} \in H_2(M',\bdry M')$,
\[
x(\hat{y}) \geq \wrap_{K'}(\hat{y})(\Delta(a,b) - 1).
\]
\end{itemize}
\end{theorem}

\begin{remark}
The first conclusion of Theorem~\ref{thm:BoundingThurstonNorm}, that $M'$ is irreducible and $\bdry$--irreducible,  essentially follows from \cite{S2}.
\end{remark}

\begin{proof}

Assume, for the moment, that either $M'$ is reducible or $\bdry$--reducible or that $K'$ is not mp-small. Then there exists an essential, connected, properly embedded planar surface $Q \subset N$ such that 
$\bdry Q$ has at most one component not in $T$, 
$\bdry_T Q$ is non-empty (because $N$ is irreducible and $\bdry$--irreducible), 
and every component of $\bdry_T Q$ has slope $b$. 
Let $\hat{Q} \subset M'$ be the sphere or disc that results from capping off $\bdry_T Q$ with discs. Lemma~\ref{no gabai disc} shows that there is no multi-$\bdry$-compressing disc for $\hat{Q}$.  
Then by Theorem~\ref{Previous Sutured Thm}, since either $M$ is reducible or $H_2(M, \bdry M)$ has an exceptional class with respect to $K$, we have
\[
0 > -\chi(\wihat{Q}) \geq |\wihat{Q} \cap K'|(\Delta(a,b) - 1) \geq 0
\]
which is a contradiction. Thus, $M'$ is irreducible, $\boundary$--irreducible, and $K'$ is mp-small.

Because $M'$ is irreducible and $\boundary$--irreducible, every sphere and disc in $M'$ separates. So consider a 
 class $\hat{y} \in H_2(M', \boundary M')$.  
Among the taut surfaces in $M'$ representing $\hat{y}$, let $\wihat{Q} \subset M'$ be chosen to minimize $|\wihat{Q} \cap K'|$. 
Tautness implies that no component of $\hat{Q}$ is a sphere or disc, that $x(\hat{y}) = -\chi(\wihat{Q})$, and that there is no compressing disc for $\hat{Q}$ in $M'$.  
The minimality gives $\wrap_{K'}(\hat{y}) = |\wihat{Q} \cap K'|$ while also implying that there can be no compressing disc for $Q = \wihat{Q} \cap N$ in $N$.  
Since every sphere and disc in $M'$ separates, Lemma \ref{no gabai disc} implies there are also no multi-$\boundary$-compressing discs for $Q$ with respect to $K$. 

If $\wihat{Q} \cap K' = \nil$, then $\wrap_{K'}(\hat{y}) = 0$ and the desired inequality is trivially true. Thus, assume that $\wihat{Q} \cap K' \neq \nil$. Using Theorem \ref{Previous Sutured Thm} again, we then have
\[
x(\hat{y}) = -\chi(\wihat{Q}) \geq |\wihat{Q} \cap K'|(\Delta(a,b) - 1) = \wrap_{K'}(\hat{y})(\Delta(a,b) - 1) 
\]
as desired.
\end{proof}

The next corollary is a useful specialization.  

\begin{corollary}\label{cor:twofillings}
Let $N$ be a compact, orientable, irreducible, $\boundary$--irreducible $3$--manifold such that $\boundary N$ is a union of tori. 
Given distinct slopes  $a$ and $b$ in a component $T$ of $\bdry N$, let $M=N_T(b)$ and $M'=N_T(a)$ be the results of Dehn filling along these slopes, and let $K$ and $K'$ be the core knots of these fillings respectively.

Assume $M$ and $M'$ are irreducible, $\bdry$--irreducible and $K'$ has non-zero wrapping number with respect to a class $\hat{y} \in H_2(M', \bdry M')$.  If there exists a class of $H_2(M, \bdry M)$ that is norm-degenerate with respect to $K$, then
\[
\Delta(a,b) \leq 1 + x(\hat{y})/\wrap_{K'}(\hat{y}) \leq 1+x(\hat{y}) \]
\end{corollary}

\begin{proof}
Since we may assume that both $H_2(M, \bdry M)$ and $H_2(M', \bdry M')$ are non-trivial, $N$ is not a solid torus.  By the irreducibility and $\bdry$--irreduciblity of $M$ and $M'$, every sphere and disc in $M$ and $M'$ must separate.  Thus, according to Lemma~\ref{lem:exceptionalisnormdegen} any class in $H_2(M,\bdry M)$ that is norm-degenerate with respect to $K$ is also exceptional with respect to $K$.
Then, due to Theorem~\ref{thm:BoundingThurstonNorm}, for every non-zero
$\hat{y} \in H_2(M', \bdry M')$ we have $x(\hat{y}) \geq \wrap_{K'}(\hat{y})(\Delta(a,b) - 1)$.  When the wrapping number is non-zero, we may obtain the stated inequalities.
\end{proof}

We can now bound the number of slopes producing filled manifolds with norm-reducing classes (with respect to the filling).

\begin{corollary}\label{cor:boundsondegenerateslopes}
Let $N$ be a compact, orientable, irreducible, and $\boundary$--irreducible $3$--manifold such that $\boundary N$ is a union of tori. 
Assume for $i=1,2$, there is a slope $a_i$ in the component $T$ of $\bdry N$ such that the manifold $M'_i = N_T(a_i)$  is irreducible and $\bdry$--irreducible and  the core $K'_i$ of the Dehn filling has non-zero wrapping number with respect to a class $\hat{y}_i \in H_2(M'_i, \bdry M'_i)$.  If $\Delta(a_1, a_2) >0$, then there are at most 
\[
(1 + x(\hat{y}_1))(1 + x(\hat{y}_2)) + (\Delta(a_1, a_2)-1)(1 + x(\hat{y}_1))^2
\]
slopes $b \subset T$ distinct from $a_1$ and $a_2$
 such that the $3$--manifold $N_T(b)$ obtained by filling $T$ along $b$ is irreducible, $\bdry$--irreducible, and has a norm-reducing class with respect to the filling.
\end{corollary}

\begin{proof}
By Corollary~\ref{cor:twofillings}, if $b$ is a slope in $T$ such that $N_T(b)$ is irreducible, $\bdry$--irreducible, and has a norm-degenerating slope for the core of the filling, then
\[ \Delta(a_1,b) \leq 1 + x(\hat{y}_1)  \quad \mbox{ and } \quad \Delta(a_2, b) \leq 1 + x(\hat{y}_2).\]
Then Lemma~\ref{lem:boudinglatticepoints} below gives that the number of slopes $b$ satisfying these constraints is at most 
\[
(1 + x(\hat{y}_1))(1 + x(\hat{y}_2)) + (\Delta(a_1, a_2)-1)(1 + x(\hat{y}_1))^2.
\]
\end{proof}

\begin{lemma}\label{lem:boudinglatticepoints}
Given slopes $b,c$ in $T$ with $\Delta(b,c)\geq1$ and positive numbers $B, C$, then the number of slopes $a$ in $T$ such that $\Delta(a,b) \leq B$ and $\Delta(a,c) \leq C$ is at most $BC+(\Delta(b,c)-1)B^2$.
\end{lemma}

\begin{proof}
Let us regard slopes as being represented by oriented simple closed curves. We may choose a basis for $H_1(T)$ in which $[b] = (1,0)$ and $[c] = (r,s)$ for coprime integers $0\leq r<s$.   Then $\Delta(b,c) = s$.  For any slope $a$ in $T$, we may choose an orientation of the curve so that the constraints $\Delta(a,b) \leq B$ and $\Delta(a,c) \leq C$ and the orientation restrict its representatives in this homology basis to an element of the set $\Lambda$ of integer lattice point in the trapezoid $\{(x,y) \colon |y|\leq B, |ry-sx|\leq C, x\geq0 \}$. For points $(x,y) \in \Lambda$, one deduces that 
\[ 0\leq x \leq s |x| \leq |ry-sx| + r|y| \leq C + rB \leq C+(s-1)B = C+(\Delta(b,c)-1)B.\]
Thus $|\Lambda| \leq B \cdot (C+sB) = BC+(\Delta(b,c)-1)B^2$, giving an upper bound on the number of slopes $a$ in $T$ satisfying the constraints.
\end{proof}

\begin{theorem}\label{thm:main}
Let $N$ be a compact, connected, orientable, irreducible, and $\boundary$--irreducible $3$--manifold whose boundary is a union of tori. Then either 
\begin{enumerate}
\item $N$ is a product of a torus and an interval,
\item $N$ is a cable space, or 
\item for each torus component $T \subset \bdry N$ there is a finite set of slopes $\calR=\calR(N,T)$ in $T$ such that if $b \not \in \calR$ then $b$ is not norm-reducing.
 \end{enumerate}
\end{theorem} 

\begin{proof}
Let $T$ be a particular component of $\bdry N$.  By \cite{hoffmanmatignon, GordonLueckeReducible}, 
$N_T(a)$ is a irreducible for at most three slopes $a$.   
By \cite[Corollary 2.4.4]{CGLS}, 
unless $N \cong T \times [0,1]$ or $N$ is a cable space, $N_T(a)$ is $\bdry$--irreducible for at most three slopes $a$.
Hence, we now assume $N$ is neither homeomorphic to $T \times [0,1]$ nor a cable space, so that there are at most $6$ slopes in $T$ for which $N_T(a)$ is reducible or $\bdry$--irreducible.

Let $(\bdry_T)_*\co H_2(N,\boundary N) \to H_1(T)$ be the composition of the boundary map on $H_2(N,\boundary N)$ with the projection from $H_1(\boundary N)$ to $H_1(T)$. For every slope $a$ in $T$ that generates a rank $1$ subspace of the image of $(\bdry_T)_*$ in $H_1(T)$, there is some class $\hat{y} \in H_2(N_T(a), \bdry N_T(a))$ such that $\wind_a(\hat{y}) >0$.  Since $\wind_a$ gives a lower bound on $\wrap_a$, the core of the Dehn filling $N_T(a)$ has non-zero wrapping number with respect to the class $\hat{y}$.  Therefore, if $(\bdry_T)_*$ surjects onto $H_1(T)$, the core of any Dehn filling of $N$ along $T$ will have non-zero wrapping number with respect to some class in the filled manifold.   In this case we may find a pair of slopes satisfying the hypotheses of Corollary~\ref{cor:boundsondegenerateslopes} so that the number of norm-reducing, but irreducible, and $\bdry$--irreducible slopes is finite.  Since the number of reducible or $\bdry$--reducible slopes in $T$ is also finite, we have our conclusion.

On the other hand, if $(\bdry_T)_*$ does not surject onto $H_1(T)$, its image must be a rank $1$ subspace generated by a single slope, say $b$.  For every other slope $a \neq b$, $\wind_a = 0$.  Hence for all $a \neq b$, $\rho_a$ gives an isomorphism $H_2(N_T(a), \bdry N_T(a)) \cong H_2(N, \bdry N - T)$.  Then it follows from \cite{Sela} (but using  \cite[Corollary 2.4]{G2} instead of just \cite[Theorem 1.8]{G2} to avoid hypotheses of atoroidality, see also \cite[Theorem A.21]{LackenbyDSOK}) that there are finitely many norm reducing fillings.
\end{proof}

\subsection{The dual norm}
	 As we observed in the introduction, Theorem~\ref{thm:dualnorm} shows that, in general, there are no norm-degenerating classes with respect to a knot that is surgery dual to a knot with ``large'' dual Thurston norm, quantified in terms of the distance of the surgery. 

\begin{theorem}\label{thm:dualnorm}
Assume that every sphere, disc, annulus, and torus in $M'$ separates. Given a class $\alpha \in H_1(M';\Z)$ and an integer $\Delta$, if
\[
(\Delta-1)x^*(\alpha) > 1 
\]
then no  Dehn surgery of distance $\Delta$ on a knot representing $\alpha$ produces an irreducible, $\bdry$--irreducible $3$--manifold $M$ which has a norm-degenerating class with respect to the core of the surgery.
\end{theorem}

\begin{proof}
Assume 
$
(\Delta-1)x^*(\alpha) > 1
$
so that $\Delta \geq 2$ and $x^*(\alpha)-1/(\Delta-1) > 0$.

Since $M'$ contains no non-separating sphere, disc, annulus, or torus, the Thurston norm on $M'$ is actually a norm and not just a pseudo-norm. Thus, the unit norm ball in $H_2(M',\boundary M')$ is compact and $x^*(\alpha) = \sup_{x(\tau)=1} |\alpha \cdot \tau|$. Since $x^*$ is continuous,
there exists a class $\sigma \in H_2(M', \boundary M'; \R)$ realizing this supremum, i.e.\ such that $x(\sigma) = 1$ and $x^*(\alpha) = |\alpha \cdot \sigma|$.
For any $\epsilon > 0$, there is a rational class $\hat{z}' \in H_2(M', \bdry M' ; \Q)$ approximating $\sigma$ such that $ x(\hat{z}') = 1$ and 
\[ |\alpha \cdot \sigma| \geq |\alpha \cdot \hat{z}'| > |\alpha \cdot \sigma| - \epsilon.\]
In particular, since $(\Delta-1)x^*(\alpha)>1$, let us choose $\epsilon$ so that $x^*(\alpha)-1/(\Delta-1) > \epsilon >0$.

Since $|\alpha \cdot \tau| / x(\tau)$ is constant for non-zero multiples of any non-zero class $\tau \in H_2(M,\bdry M;\R)$, there exists an integral class $\hat{z} \in H_2(M,\bdry M;\Z)$ that is a positive multiple of the rational class $\hat{z}'$ for which 
\[ |\alpha \cdot \sigma| \geq \frac{|\alpha \cdot \hat{z}|}{x(\hat{z})} > |\alpha \cdot \sigma| - \epsilon.\]
Being an integral class, $\hat{z}$ is represented by a surface.  For any taut surface $\hat{Q}$ representing $\hat{z}$ we have $x(\hat{z})=-\chi(\hat{Q})$ and $|\alpha \cdot \hat{z}| = \wind_\alpha (\hat{Q})$. 

Now let $K'$ be any knot representing $\alpha$.
Among the taut surfaces representing $\hat{z}$, choose $\hat{Q}$ to be one that  minimizes $|\hat{Q} \cap K'|$.
Thus $\wrap_{K'}(\hat{Q}) \geq \wind_{K'}(\hat{Q})  = |K' \cdot \hat{Q}| = |\alpha \cdot \hat{z}|$.

 Hence  by the choice of $\sigma$,
\[ x^*(\alpha)\geq \frac{\wind_\alpha (\hat{Q})}{-\chi(\hat{Q})} >x^*(\alpha) - \epsilon. \tag{\markd}\]

Since  $x^*(\alpha)-1/(\Delta-1) \geq \epsilon > 0$, we have $(\Delta-1)(x^*(\alpha) - \epsilon) \geq 1$ and thus the right hand inequality of (\markd) gives
\[(\Delta-1)\frac{\wind_\alpha (\hat{Q})}{-\chi(\hat{Q})} >(\Delta-1)(x^*(\alpha) - \epsilon) \geq 1. \]
Consequently,
\[
(\Delta - 1)|K' \cap \hat{Q}| 
= (\Delta - 1) \wrap_{K'}(\hat{Q})
 \geq (\Delta - 1)\wind_\alpha(\hat{Q}) > - \chi(\hat{Q})
\]
By the choice of $\hat{Q}$ and Lemma \ref{no gabai disc}, there is no multi-$\boundary$-compressing disc for $\wihat{Q}$. Thus, by Theorem \ref{Previous Sutured Thm}, if $M$ is obtained by a distance $\Delta$ Dehn surgery on $K'$, then $H_2(M, \boundary M)$ cannot contain a norm-degenerating class with respect to the core of the surgery.
\end{proof}

\section{Genus growth in twist families.}

Let $Y$ be a closed, compact, connected, oriented, irreducible, $3$--manifold with $H_2(Y) = 0$.
Let $\{K_n\}$ be a twist family of null-homologous knots in $Y$ obtained by twisting a null-homologous knot $K=K_0$ along an unknot $c$.
That is, $K_n$ is the knot in $Y=Y_c(-1/n)$ obtained by $-1/n$--surgery on $c$ for each integer $n$.  Let $g(K_n)$ be the Seifert genus of $K_n$ and set $\omega = |\lk(K, c)|$.

\begin{theorem}\label{thm:twistfamily}
	If $|\lk(K,c)|>0$, then $\lim\limits_{n\to \infty} g(K_n) =\infty$ unless $c$ is a meridian of $K$.
\end{theorem}

\begin{proof}
	This follows as a corollary of the more precise Theorem~\ref{thm:twistfamilygrowth} below which implies the limit is finite only if $\omega x([D]) = 0$.  Here $x$ is the Thurston norm on the exterior of the link $K \cup c$ and $[D]$ is the homology class of a disk bounded by $c$, intersected by $K$, and restricted to this exterior.  Since $\omega = |\lk(K,c)| >0$, the limit is finite only if $x([D]) = 0$.  This however implies that $D$ is an annulus and hence $c$ is a meridian of $K$.
\end{proof}

	Let $N=Y-\nbhd(K\cup c)$ be the exterior of the link $K\cup c$ with boundary components $T_K$ and $T_c$ corresponding to $K$ and $c$ respectively, and use the standard associated meridian-longitude bases relative to $K$ and $c$ for these tori.    Then the exterior of $K_n$ is the manifold  $Y-\nbhd(K_n)=N_{T_c}(-1/n)$ which results from Dehn filling $N$ along the slope $-1/n$ in $T_c$; let $c_n$ be the core of this filling, setting $c=c_0$.

Let $\hat{D}$ be a disk bounded by $c$ that is transverse to $K$ and set $D = \hat{D} \cap N$.
Let $\hat{F}_n$ be a Seifert surface for $K_n$ that is transverse to $c_n$ and set $F_n = \hat{F}_n \cap N$.

\begin{lemma}\label{lem:homologyoftwists}
$[F_{n+1}] = [F_n] + \omega [D]$ for all integers $n$.
\end{lemma}

\begin{proof}
	Since $Y$ is a rational homology sphere by assumption, each knot $K_n$ (and $c$) has a unique homology class of Seifert surface up to sign.  The formula then follows since $\omega = |\lk(K,c)|$ and the surfaces $F_n$ and $D$ are the restrictions of Seifert surfaces for $K_n$ and $c$ to $N$.  
	Indeed, $\bdry [F_n]$ is homologous to one longitude of slope $n \omega$ in $T_K$ and $\omega$ parallel curves of slope $-1/n$ in $T_c$ while $\bdry [D]$ is homologous to $\omega$ meridians in $T_K$ and one longitude of slope $0$ in $T_c$.  It follows that (heeding orientations) $[F_n]+\omega [D]$ is represented by  a properly embedded surface in $N$ that is the Haken sum of $F_n$ and $\omega$ parallel copies of $D$ which has boundary homologous to that of $\bdry [F_{n+1}]$. 
	If $[F_{n+1}]-[F_n]-\omega[D]$ were a non-zero class, it would be represented by a boundaryless surface in $N$ and thus represent a non-zero class in $H_2(Y)$ --- a contradiction.
		Hence $[F_{n+1}] =  [F_n] + \omega [D]$. 
\end{proof}

\begin{theorem}\label{thm:twistfamilygrowth}
	There is a constant $G=G(K,c)$ such that $2g(K_n) = 2G+n \omega  x([D])$ for sufficiently large $n>0$. 
\end{theorem}

\begin{proof}	
	Among discs bounded by $c$ in $Y$, let $\hat{D}$ be one for which $|K \cap \hat{D}| = p>0$ is minimized and set $D = \hat{D} \cap N$.  Note that the minimality implies the punctured disc $D$ is incompressible and $\bdry$--incompressible. Moreover $\bdry D$ consists of one longitude of $c$ and $p$ meridional curves of $K$.  In particular, if $p=1$ then $D$ is an annulus so that $x([D]) = 0$ and $c$ is a meridian of $K$.  Hence $K=K_n$ for all integers $n$ so the genus is constant and the theorem holds.    Thus we assume $p\geq 2$.  This further implies that $N$ is not the product of a torus and an interval.
	
	If $N$ is a cable space, since $D$ is not an annulus but is a properly embedded, non-separating, incompressible and $\bdry$--incompressible surface, it must be a fiber in a fibration of $N$ over $S^1$.  (All classes in $H_2(N, \bdry N;\Z)$ other than multiples of the class of the cabling annulus are represented by fibers.)  Therefore because $\bdry D$ consists of a longitude of $c$ and meridians of $K$, it follows that $Y \cong S^3$ and $K$ is a torus knot in the solid torus exterior of the unknot $c$.  In particular, this means that for some  integer $q$ coprime to $p=|K \cap \hat{D}|$, the knot $K_n$ is the $(p,q+np)$--torus knot and the theorem holds.  Therefore we may assume that $N$ is not a cable space.
	
	If $N$ is reducible, then there is a sphere in $N$ that does not bound a ball in $N$ and yet must bound a ball in $Y$ that contains either $K$ or $c$. If this sphere separates the two components of $\bdry N$ then it separates $K$ and $c$ in $Y$ implying that $\lk(K,c)=0$, contrary to assumption.  Thus $K \cup c$ must be contained in a ball in $Y$ and may be viewed as being contained in an $S^3$ summand of $Y$. Thus $N=N'\#Y$ where $N'$ is the irreducible exterior of $K \cup c$ in $S^3$.   Since the summand will not affect the genera of the knots $K_n$, we may run the argument for $K \cup c$ in $S^3$.  Thus we may assume $N$ is irreducible.

	Let $\hat{z}_n$ be the homology class of an oriented Seifert surface for $K_n$ in $Y-\nbhd(K_n)$ for which $x(\hat{z_n}) = 2g(K_n)-1$.  Then set $z_n = \rho_{-1/n}(\hat{z}_n)$  to be the homology class of the restriction of the Seifert surface to $N = Y -\nbhd(K \cup c)$.  By Theorem~\ref{thm:main}, there is a finite set of integers $\calR$ such that 
	\[x(z_n) = x(\hat{z}_n) + \wind_{K_n}(\hat{z}_n) \]
	if $n \not \in \calR$.  Since $\omega = \wind_{K_n}(\hat{z}_n)$ for all integers $n$ and $2g(K_n)-1 = x(\hat{z}_n)$,  then 
	when $n \gg 0$ we have
	\[2(g(K_{n+1}) - g(K_n)) = x(z_{n+1}) - x(z_n) = x(z_{n+1} - z_n).\]
	By Lemma~\ref{lem:homologyoftwists}, $z_{n+1}-z_n = \omega [D]$ for all integers $n$.
	Hence for $n \gg 0$, $2(g(K_{n+1}) - g(K_n)) =  \omega x([D])$.  
	Therefore when $n$ is sufficiently large, $2g(K_n) = 2G + n \omega x([D])$ for some constant $G$ as desired.
\end{proof}

\begin{remark}
At the expense of having to reckon with multiple homology classes of Seifert surfaces, one should be able to prove Theorem~\ref{thm:twistfamilygrowth} without the hypothesis that $Y$ is a rational homology sphere.
\end{remark}

\begin{remark}
One ought to be able to prove Theorem~\ref{thm:main} and
Theorem~\ref{thm:twistfamilygrowth} using link Floer Homology.
\end{remark}

\section{Acknowledgements}

The authors would like to thank Colby College and University of Miami for their hospitality during this project and K. Motegi for inspiring conversations. This work was partially supported by a grant from the Simons Foundation (\#209184 to Kenneth L.\ Baker).

\end{document}